\documentclass[12pt]{amsart}
\usepackage{amscd,amssymb}
\usepackage{amsthm,amsmath,amssymb}
\usepackage[matrix,arrow]{xy}
\usepackage{enumerate}

\newtheorem{theorem}{Theorem}[section]

\newtheorem{lemma}[theorem]{Lemma}
\newtheorem{corollary}[theorem]{Corollary}
\newtheorem{conjecture}[theorem]{Conjecture}
\newtheorem*{conjecture*}{Conjecture}
\newtheorem{problem}[theorem]{Problem}

\theoremstyle{definition}
\newtheorem{example}[theorem]{Example}
\newtheorem{definition}[theorem]{Definition}

\theoremstyle{remark}

\makeatletter\@addtoreset{equation}{section} \makeatother

\author{Ivan Cheltsov}

\title{On a conjecture of Hong and Won}


\begin{document}

\begin{abstract}
We give an explicit counter-example to a conjecture of Kyusik Hong and Joonyeong~Won about $\alpha$-invariants of polarized smooth del Pezzo surfaces of degree one.
\end{abstract}

\sloppy

\maketitle

All varieties are assumed to be algebraic, projective and defined over $\mathbb{C}$.

\section{Introduction}
\label{section:intro}

In \cite{Tian}, Tian defined the $\alpha$-invariant of a smooth Fano variety and proved

\begin{theorem}[{\cite{Tian}}]
\label{theorem:Tian}
Let $X$ be a smooth Fano variety of dimension $n$.
Suppose that
$$
\alpha(X)>\frac{n}{n+1}.
$$
Then $X$ admits a K\"ahler--Einstein metric.
\end{theorem}

Two-dimensional smooth Fano varieties are also known as smooth del Pezzo surfaces.
The possible values of their $\alpha$-invariants are given by

\begin{theorem}[{\cite[Theorem~1.7]{Ch08}}]
\label{theorem:Cheltsov}
Let $S$ be a smooth del Pezzo surface. Then
$$
\alpha(S)=\left\{%
\aligned
&\frac{1}{3}\ \mathrm{if}\ S\cong\mathbb{F}_{1}\ \mathrm{or}\ K_{S}^{2}\in\{7,9\},\\%
&\frac{1}{2}\ \mathrm{if}\ S\cong\mathbb{P}^{1}\times\mathbb{P}^{1}\ \mathrm{or}\ K_{S}^{2}\in\{5,6\},\\%
&\frac{2}{3}\ \mathrm{if}\ K_{S}^{2}=4,\\%
&\frac{2}{3}\ \mathrm{if}\ S\ \mathrm{is\ a\ cubic\ surface\ in}\ \mathbb{P}^{3}\ \mathrm{with\ an\ Eckardt\ point},\\%
&\frac{3}{4}\ \mathrm{if}\ S\ \mathrm{is\ a\ cubic\ surface\ in}\ \mathbb{P}^{3}\ \mathrm{without\ Eckardt\ points},\\%
&\frac{3}{4}\ \mathrm{if}\ K_{S}^{2}=2\ \mathrm{and}\ |-K_{S}|\ \mathrm{has\ a\ tacnodal\ curve},\\%
&\frac{5}{6}\ \mathrm{if}\ K_{S}^{2}=2\ \mathrm{and}\ |-K_{S}|\ \mathrm{has\ no\ tacnodal\ curves},\\%
&\frac{5}{6}\ \mathrm{if}\ K_{S}^{2}=1\ \mathrm{and}\ |-K_{S}|\ \mathrm{has\ a\ cuspidal\ curve},\\%
&1\ \mathrm{if}\ K_{S}^{2}=1\ \mathrm{and}\ |-K_{S}|\ \mathrm{has\ no\ cuspidal\ curves}.\\%
\endaligned\right.%
$$
\end{theorem}

Let $X$ be an arbitrary smooth algebraic variety, and let $L$ be an ample $\mathbb{Q}$-divisor on it.
In~\cite{Tian2012},~Tian defined a new invariant $\alpha(X,L)$ that generalizes the classical $\alpha$-invariant.
If~$X$ is a smooth Fano variety, then $\alpha(X)=\alpha(X,-K_X)$.
By~\cite[Theorem~A.3]{ChSh}, one has
$$
\alpha\big(X,L\big)=\mathrm{sup}\left\{\lambda\in\mathbb{Q}\ \Bigg|%
\aligned
&\text{the log pair}\ \left(X, \lambda D\right)\ \text{is log canonical}\\
&\text{for every effective $\mathbb{Q}$-divisor}\ D\sim_{\mathbb{Q}} L
\endaligned\Bigg.\right\}\in\mathbb{R}_{>0}.%
$$

In \cite{Dervan1}, Dervan proved

\begin{theorem}[{\cite[Theorem~1.1]{Dervan1}}]
\label{theorem:Dervan}
Suppose that
$-K_X-\frac{n}{n+1}\frac{-K_X\cdot L^{n-1}}{L^n}L$
is nef, and
$$
\alpha\big(X,L\big)>\frac{n}{n+1}\frac{-K_X\cdot L^{n-1}}{L^n}.
$$
Then the pair $(X,L)$ is $K$-stable.
\end{theorem}

Donaldson, Tian and Yau conjectured that the following conditions are equivalent:
\begin{itemize}
\item the pair $(X,L)$ is $K$-polystable,
\item the variety $X$ admits a constant scalar curvature K\"ahler metric in $\mathrm{c}_1(L)$.
\end{itemize}
In \cite{CDS}, this conjecture has been proved in the case when $X$ is a Fano variety and $L=-K_X$.
Therefore, Dervan's Theorem~\ref{theorem:Dervan} is a generalization of Tian's Theorem~\ref{theorem:Tian}.

For smooth del Pezzo surfaces, Theorem~\ref{theorem:Dervan} gives

\begin{theorem}[{\cite{HongWon,ChJ}}]
\label{theorem:stable}
Let $S$ be a smooth del Pezzo surface such that $K_S^2=1$ or $K_S^2=2$.
Let $A$ be an ample $\mathbb{Q}$-divisor on the surface $S$ such that the divisor
$$
-K_S-\frac{2}{3}\frac{-K_S\cdot A}{A^2}A
$$
is nef. Then the pair $(S,A)$ is $K$-stable.
\end{theorem}

This result is closely related to

\begin{problem}[{cf. Theorem~\ref{theorem:Cheltsov}}]
\label{problem:del-Pezzo}
Let $S$ be a smooth del Pezzo surface.
Compute
$$
\alpha(S,A)\in\mathbb{R}_{>0}
$$
for every ample $\mathbb{Q}$-divisor $A$ on the surface $S$.
\end{problem}

Hong and Won suggested an answer to Problem~\ref{problem:del-Pezzo} for del Pezzo surfaces of degree one.
This answer is given by their \cite[Conjecture~4.3]{HongWon}, which is Conjecture~\ref{conjecture:Hong-Won} in Section~\ref{section:conjecture}.

The main result of this paper is

\begin{theorem}[{cf. Theorem~\ref{theorem:Cheltsov}}]
\label{theorem:main}
Let $S$ be a smooth del Pezzo surface such that $K_S^2=1$.
Let $C$ be an irreducible smooth curve in $S$ such that $C^2=-1$.
Then~there is a unique~curve
$$
\widetilde{C}\in\big|-2K_S-C\big|.
$$
The curve $\widetilde{C}$ is also irreducible and smooth. One has $\widetilde{C}^2=-1$ and $1\leqslant |C\cap\widetilde{C}|\leqslant C\cdot\widetilde{C}=3$.
Let $\lambda$ be a rational number such that $0\leqslant \lambda<1$.
Then $-K_S+\lambda C$ is ample and
$$
\alpha\big(S,-K_S+\lambda C\big)=\left\{%
\aligned
&\mathrm{min}\Big(\alpha(S),\frac{2}{1+2\lambda}\Big)\ \mathrm{if}\ |C\cap\widetilde{C}|\geqslant 2,\\%
&\mathrm{min}\Big(\alpha(S),\frac{4}{3+3\lambda}\Big)\ \mathrm{if}\ |C\cap\widetilde{C}|=1.\\%
\endaligned\right.
$$
\end{theorem}

Theorem~\ref{theorem:main} implies that \cite[Conjecture~4.3]{HongWon}~is~wrong. To be precise, this follows from

\begin{example}
\label{example:main}
Let $S$ be a surface in $\mathbb{P}(1,1,2,3)$ that is given by
$$
w^2=z^3+zx^2+y^6,
$$
where $x$, $y$, $z$, $w$ are coordinates such that $\mathrm{wt}(x)=\mathrm{wt}(y)=1$, $\mathrm{wt}(z)=2$ and $\mathrm{wt}(w)=3$.
Then $S$ is a smooth del Pezzo surface and $K_S^2=1$.
Let $C$ be the curve in $X$ given by
$$
z=w-y^3=0.
$$
Similarly, let $\widetilde{C}$ be the curve in $S$ that is given by $z=w+y^3=0$. Then $C+\widetilde{C}\sim -2K_S$.
Both~curves $C$ and $\widetilde{C}$ are smooth rational curves such that $C^2=\widetilde{C}^2=-1$ and $|C\cap\widetilde{C}|=1$.
All singular curves in $|-K_S|$ are nodal. Then $\alpha(S)=1$ by Theorems~\ref{theorem:Cheltsov}, so that
$$
\alpha\big(S,-K_S+\lambda C\big)=\mathrm{min}\Big(1,\frac{4}{3+3\lambda}\Big)%
$$
by Theorem~\ref{theorem:main}. But \cite[Conjecture~4.3]{HongWon} says that $\alpha(S,-K_S+\lambda C)=\mathrm{min}(1,\frac{2}{1+2\lambda})$.
\end{example}

Theorem~\ref{theorem:main} has two applications. By Theorem~\ref{theorem:Dervan}, it implies

\begin{corollary}[{\cite[Theorem~1.2]{Dervan1}}]
\label{corollary:Dervan-degree-1}
Let $S$ be a smooth del Pezzo surface such that $K_S^2=1$.
Let $C$ be an irreducible smooth curve in $S$ such that $C^2=-1$. Fix $\lambda\in\mathbb{Q}$ such that
$$
3-\sqrt{10}\leqslant \lambda\leqslant\frac{\sqrt{10}-1}{9}.
$$
Then the pair $(S,-K_S+\lambda C)$ is $K$-stable.
\end{corollary}

By \cite[Remark~1.1.3]{CheltsovParkWon3}, Theorem~\ref{theorem:main} implies

\begin{corollary}
\label{corollary:cylinders}
Let $S$ be a smooth del Pezzo surface. Suppose that $K_S^2=1$ and $\alpha(S)=1$.
Let $C$ be an irreducible smooth curve in $S$ such that $C^2=-1$. Fix $\lambda\in\mathbb{Q}$ such that
$$
-\frac{1}{4}\leqslant \lambda\leqslant\frac{1}{3}.
$$
Then $S$ does not contain $(-K_S+\lambda C)$-polar cylinders (see \cite[Definition 1.2.1]{CheltsovParkWon3}).
\end{corollary}

Corollary~\ref{corollary:Dervan-degree-1} follows from Theorem~\ref{theorem:stable}. Corollary~\ref{corollary:cylinders} follows from \cite[Theorem 2.2.3]{CheltsovParkWon3}.

Let us describe the structure of this paper.
In~Section~\ref{section:conjecture}, we describe \cite[Conjecture~4.3]{HongWon}.
In~Section~\ref{section:pairs}, we present several well known local results about singularities of log pairs.
In~Section~\ref{section:eight-lemmas}, we prove eight local lemmas that are crucial for the proof of Theorem~\ref{theorem:main}.
In~Section~\ref{section:main}, we prove Theorem~\ref{theorem:main} using Lemmas~\ref{lemma:local-1}, \ref{lemma:local-2}, \ref{lemma:local-3}, \ref{lemma:local-4}, \ref{lemma:local-5}, \ref{lemma:local-6}, \ref{lemma:local-7}, \ref{lemma:local-8}.

\bigskip

{\bf Acknowledgements.}
The main result of this paper was proved in the summer 2016.
This~happened during the 35th workshop on Geometric Methods in Physics in Bialowieza.
We would like to thank the organizers of this workshop for creative working conditions.

\section{Conjecture of Hong and Won}
\label{section:conjecture}

Let $S$ be a smooth del Pezzo surface, and let $A$ be an ample $\mathbb{Q}$-divisor on~$S$. Put
$$
\mu=\mathrm{inf}\left\{\lambda\in\mathbb{Q}_{>0}\ \Big|\ \text{the $\mathbb{Q}$-divisor}\ K_{S}+\lambda A\ \text{is pseudo-effective}\right\}\in\mathbb{Q}_{>0}.%
$$
Then $K_S+\mu A$ is contained in the boundary of the Mori cone $\overline{\mathbb{NE}}(S)$ of the surface $S$.

Suppose that $K_S^2=1$. Then $\overline{\mathbb{NE}}(S)$ is polyhedral and is generated by $(-1)$-curves in~$S$.
By a $(-1)$-curve, we mean a smooth irreducible rational curve $E\subset S$ such that $E^2=-1$.

Let $\Delta_{A}$ be the smallest extremal face of the Mori cone $\overline{\mathbb{NE}}(S)$ that contains $K_{S}+\mu A$.
Let $\phi\colon S\to Z$ be the contraction given by the face $\Delta_{A}$. Then
\begin{itemize}
\item either $\phi$ is a birational  morphism and $Z$ is a smooth del Pezzo surface,
\item or $\phi$ is a conic bundle and $Z\cong\mathbb{P}^1$.
\end{itemize}

If $\phi$ is birational and $Z\not\cong\mathbb{P}^1\times\mathbb{P}^1$, we call $A$ a divisor of $\mathbb{P}^2$-type. In this case, we have
$$
K_{S}+\mu A\sim_{\mathbb{Q}}\sum_{i=1}^8a_iE_i,
$$
where $E_1$, $E_2$, $E_3$, $E_4$, $E_5$, $E_6$, $E_7$, $E_8$ are eight disjoint $(-1)$-curves in our surface $S$,
and~$a_1$, $a_2$, $a_3$, $a_4$, $a_5$, $a_6$, $a_7$, $a_8$ are non-negative rational numbers such that
$$
1>a_1\geqslant a_2\geqslant a_3\geqslant a_4\geqslant a_5\geqslant a_6\geqslant a_7\geqslant a_8\geqslant 0.
$$
In this case, we put $s_A=a_2+a_3+a_4+a_5+a_6+a_7+a_8$.

If our ample divisor $A$ is not a divisor of $\mathbb{P}^2$-type, then the surface $S$ contains a smooth irreducible rational curve $C$ such that $C^2=0$ and
$$
K_{S}+\mu A\sim_{\mathbb{Q}}\delta C+\sum_{i=1}^7a_iE_i,
$$
where $E_1$, $E_2$, $E_3$, $E_4$, $E_5$, $E_6$, $E_7$ are disjoint $(-1)$-curves in $S$ that are disjoint from~$C$,
and~$\delta$, $a_1$, $a_2$, $a_3$, $a_4$, $a_5$, $a_6$, $a_7$ are non-negative rational numbers such that
$$
1>a_1\geqslant a_2\geqslant a_3\geqslant a_4\geqslant a_5\geqslant a_6\geqslant a_7\geqslant 0.
$$
In this case, let $\psi\colon S\to\overline{S}$ be the contraction of the curves $E_1$, $E_2$, $E_3$, $E_4$, $E_5$, $E_6$, $E_7$,
and let $\eta\colon S\to\mathbb{P}^1$ be a conic bundle given by $|C|$.
Then either $\overline{S}\cong\mathbb{F}_1$ or $\overline{S}\cong\mathbb{P}^1\times\mathbb{P}^1$.
In both cases, there exists a commutative diagram
$$
\xymatrix{
S\ar@{->}[rd]_\eta\ar@{->}[rr]^\psi&&\overline{S}\ar@{->}[ld]^\pi\\
&\mathbb{P}^1&}
$$
where $\pi$ is a natural projection. Then $\delta>0$ $\iff$ $\phi$ is a conic bundle and~$\phi=\eta$.
Similarly, if $\phi$ is birational and $Z\cong\mathbb{P}^1\times\mathbb{P}^1$, then $\delta=0$, $a_7>0$, and $\phi=\psi$. Then
\begin{itemize}
\item we call $A$ a divisor of $\mathbb{F}_1$-type in the case when $\overline{S}\cong\mathbb{F}_1$,
\item we call $A$ a divisor of $\mathbb{P}^1\times\mathbb{P}^1$-type in the case when $\overline{S}\cong\mathbb{P}^1\times\mathbb{P}^1$.
\end{itemize}
In both cases, we put $s_A=a_2+a_3+a_4+a_5+a_6+a_7$.

In order to study $\alpha(S,A)$, we may assume that $\mu=1$, because
$$
\alpha\big(S,A\big)=\frac{\alpha\big(S,\mu A\big)}{\mu}.
$$

If $A$ is a divisor of $\mathbb{P}^2$-type, let us define a number $\alpha_c(S,A)$ as follows:
\begin{itemize}
\item if $s_A>4$, we put $\alpha_c(S,A)=\frac{1}{2+a_1}$,
\item if $4\geqslant s_A>1$, we let $\alpha_c(S,A)$ to be
$$
\mathrm{max}\Bigg(\frac{2}{2+2a_1+s_A-a_2-a_3},\frac{4}{3+4a_1+2s_A-a_2-a_3-a_4},\frac{3}{2+3a_1+s_A}\Bigg),
$$
\item if $1\geqslant s_A$, we put $\alpha_c(S,A)=\mathrm{min}(\frac{2}{1+2a_1+s_A},1)$.
\end{itemize}

Similarly, if $A$ is a divisor of $\mathbb{F}_1$-type, we define $\alpha_c(S,A)$ as follows:
\begin{itemize}
\item if $s_A>4$, we put $\alpha_c(S,A)=\frac{1}{2+a_1+\delta}$,
\item if $4\geqslant s_A>1$, we let $\alpha_c(S,A)$ to be
$$
\mathrm{max}\Bigg(\frac{2}{2+2a_1+s_A-a_2-a_3+2\delta},\frac{4}{3+4a_1+2s_A-a_2-a_3-a_4+4\delta},\frac{3}{2+3a_1+s_A+3\delta}\Bigg),
$$
\item if $1\geqslant s_A$, we put $\alpha_c(S,A)=\mathrm{min}(\frac{2}{1+2a_1+s_A+2\delta},1)$.
\end{itemize}

Finally, if $A$ is a divisor of $\mathbb{P}^1\times\mathbb{P}^1$-type, we define $\alpha_c(S,A)$ as follows:
\begin{itemize}
\item if $s_A>4$, we put $\alpha_c(S,A)=\frac{1}{2+a_1+\delta}$,
\item if $4\geqslant s_A>1$, we let $\alpha_c(S,A)$ to be
$$
\mathrm{max}\Bigg(\frac{2}{2+s_{A}-a_7-a_2-a_3+2\delta},\frac{4}{3+2s_{A}-2a_7-a_2-a_3-a_4+4\delta},\frac{3}{2+s_{A}-a_7+3\delta}\Bigg),
$$
\item if $1\geqslant s_A$, we put $\alpha_c(S,A)=\mathrm{min}(\frac{2}{1+s_{A}-a_7+2\delta},1)$.
\end{itemize}

The conjecture of Hong and Won is

\begin{conjecture}[{\cite[Conjecture~4.3]{HongWon}}]
\label{conjecture:Hong-Won}
If $\alpha(S)=1$, then $\alpha(S,A)=\alpha_c(S,A)$.
\end{conjecture}

The main evidence for this conjecture is

\begin{theorem}[{\cite{HongWon}}]
\label{theorem:Hong-Won}
Let $D$ be an effective $\mathbb{Q}$-divisor on the surface $S$ such~that $D\sim_{\mathbb{Q}} A$.
Then the log pair $(S,\alpha_c(S,A)D)$ is log canonical outside of finitely many points.
\end{theorem}

As we already mentioned in Section~\ref{section:intro}, Example~\ref{example:main} shows that Conjecture~\ref{conjecture:Hong-Won} is wrong.
However, the smooth del Pezzo surface of degree one in Example~\ref{example:main} is rather special.
Therefore, Conjecture~\ref{conjecture:Hong-Won} may hold for \emph{general} smooth del Pezzo surfaces of degree one.

By \cite[Remark~1.1.3]{CheltsovParkWon3}, it follows from Conjecture~\ref{conjecture:Hong-Won} that $S$ does not contain $A$-polar cylinders (see \cite[Definition 1.2.1]{CheltsovParkWon3}) when $\alpha(S)=1$ and $a_1$ and $\delta$ are small enough.

\section{Singularities of log pairs}
\label{section:pairs}

Let $S$ be a smooth surface, and let $D$ be an effective $\mathbb{Q}$-divisor on it. Write
$$
D=\sum_{i=1}^{r}a_iC_i
$$
where each $C_i$ is an irreducible curve on $S$, and each $a_i$ is a non-negative rational number.
We assume here that all curves $C_1,\ldots,C_r$ are different.

Let $\gamma\colon\mathcal{S}\to S$ be a birational morphism such that the surface $\mathcal{S}$ is smooth as well.
It~is~well-known that the morphism $\gamma$ is a composition of $n$ blow ups of smooth points.
Thus, the morphism $\gamma$ contracts $n$ irreducible curves. Denote these curves by $\Gamma_1,\ldots, \Gamma_n$.
For each curve $C_i$, denote by $\mathcal{C}_i$ its proper transform on the surface $\mathcal{S}$.
Then
$$
K_{\mathcal{S}}+\sum_{i=1}^{r}a_i\mathcal{C}_i+\sum_{j=1}^{n}b_j\Gamma_j\sim_{\mathbb{Q}}\gamma^{*}\big(K_{S}+D\big)
$$
for some rational numbers $b_1,\ldots,b_n$. Suppose, in addition,
that the divisor
$$
\sum_{i=1}^r\mathcal{C}_i+\sum_{j=1}^n\Gamma_j
$$
has simple normal crossing singularities. Fix a point $P\in S$.

\begin{definition}
\label{definition:lct}
The log pair $(S,D)$ is \emph{log canonical} (respectively \emph{Kawamata log terminal}) at the point $P$ if
the following two conditions are satisfied:
\begin{itemize}
\item $a_i\leqslant 1$ (respectively $a_i<1$) for every $C_i$ such that $P\in C_i$,
\item $b_j\leqslant 1$ (respectively $b_j<1$) for every $\Gamma_j$ such that $\pi(\Gamma_j)=P$.
\end{itemize}
\end{definition}

This definition does not depend on the choice of the birational morphism~$\gamma$.

The log pair $(S,D)$ is said to be \emph{log canonical} (respectively \emph{Kawamata log terminal}) if~it is log
canonical (respectively, \emph{Kawamata log terminal}) at every point in $S$.

The following result follows from Definition~\ref{definition:lct}. But it is very handy.

\begin{lemma}
\label{lemma:convexity}
Suppose that the singularities of the pair $(S,D)$ are not log canonical at $P$.
Let~$D^{\prime}$~be~an effective $\mathbb{Q}$-divisor on $S$ such that $(S,D^{\prime})$ is log canonical at~$P$ and $D^{\prime}\sim_{\mathbb{Q}} D$.
Then there exists an effective $\mathbb{Q}$-divisor $D^{\prime\prime}$ on the surface $S$ such that
$$
D^{\prime\prime}\sim_{\mathbb{Q}} D,
$$
the log pair $(S,D^{\prime\prime})$ is not log canonical at $P$, and $\mathrm{Supp}(D^{\prime})\not\subseteq\mathrm{Supp}(D^{\prime\prime})$.
\end{lemma}

\begin{proof}
Let $\epsilon$ be the largest rational number such that $(1+\epsilon) D-\epsilon D^{\prime}$ is effective. Then
$$
(1+\epsilon) D-\epsilon D^{\prime}\sim_{\mathbb{Q}} D.
$$
Put $D^{\prime\prime}=(1+\epsilon) D-\epsilon D^{\prime}$. Then $(S,D^{\prime\prime})$ is not log canonical at $P$, because
$$
D=\frac{1}{1+\epsilon}D^{\prime\prime}+\frac{\epsilon}{1+\epsilon}D^{\prime}.
$$
Furthermore, we have $\mathrm{Supp}(D^{\prime})\not\subseteq\mathrm{Supp}(D^{\prime\prime})$ by construction.
\end{proof}

Let $f\colon \widetilde{S}\to S$ be a blow up of the point $P$. Let us denote the $f$-exceptional curve~by~$F$.
Denote by $\widetilde{D}$ the proper transform of the divisor $D$ via $f$. Put $m=\mathrm{mult}_{P}(D)$.

\begin{theorem}[{\cite[Exercise~6.18]{CoKoSm}}]
\label{theorem:Skoda} If $(S,D)$ is not log canonical at $P$, then $m>1$.
\end{theorem}

Let $C$ be an irreducible curve in the surface $S$. Suppose that $P\in C$ and $C\not\subseteq\mathrm{Supp}(D)$.
Denote by $\widetilde{C}$ the proper transform of the curve $C$ via $f$.
Fix $a\in\mathbb{Q}$ such that $0\leqslant a\leqslant 1$.
Then $(S, aC+D)$ is not log canonical at $P$ if and only if the log pair
\begin{equation}
\label{equation:log-pair-2}
\Bigg(\widetilde{S},a\widetilde{C}+\widetilde{D}+\Big(a\mathrm{mult}_{P}\big(C\big)+m-1\Big)F\Bigg)
\end{equation}
is not log canonical at some point in $F$. This follows from Definition~\ref{definition:lct}.

\begin{theorem}[{\cite[Exercise~6.31]{CoKoSm}}]
\label{theorem:adjunction}
Suppose that $C$ is smooth at $P$, and $(D\cdot C)_{P}\leqslant 1$.
Then the log pair $(S, aC+D)$ is log canonical at $P$.
\end{theorem}

\begin{corollary}
\label{corollary:adjunction}
Suppose that the log pair \eqref{equation:log-pair-2} is not log canonical at some point in $F\setminus\widetilde{C}$.
Then either $a\mathrm{mult}_{P}(C)+m>2$ or $m>1$ (or both).
\end{corollary}

Let us give another application of Theorem~\ref{theorem:adjunction}.

\begin{lemma}
\label{lemma:dP2}
Suppose that there is a double cover $\pi\colon S\to\mathbb{P}^2$ branched in a curve $R\subset\mathbb{P}^2$.
Suppose also that $(S,D)$ is not log canonical at $P$, and $D\sim_{\mathbb{Q}}\pi^*(\mathcal{O}_{\mathbb{P}^2}(1))$.
Then $\pi(P)\in R$.
\end{lemma}

\begin{proof}
The log pair $(\widetilde{S},\widetilde{D}+(m-1)F)$ is not log canonical at some point $Q\in F$.
Then
\begin{equation}
\label{equation:22}
m+\mathrm{mult}_{Q}\big(\widetilde{D}\big)>2
\end{equation}
by Theorem~\ref{theorem:Skoda}.
Suppose that $\pi(P)\not\in R$. Then there is $Z\in|\pi^*(\mathcal{O}_{\mathbb{P}^2}(1))|$ such that
\begin{itemize}
\item the curve $Z$ passes through the point $P$,
\item  the proper transform of the curve $Z$ on the surface $\widetilde{S}$ contains $Q$.
\end{itemize}
Denote by $\widetilde{Z}$ the proper transform of the curve $Z$ on the surface $\widetilde{S}$.

By Lemma~\ref{lemma:convexity}, we may assume that
the support of the $\mathbb{Q}$-divisor $D$ does not contain at least one irreducible component of the curve $Z$,
because $(S,Z)$ is log canonical at $P$.
Thus, if $Z$ is irreducible, then
$2-m=\widetilde{Z}\cdot\widetilde{D}\geqslant\mathrm{mult}_{Q}(\widetilde{D})$,
which contradicts \eqref{equation:22}.

We see that $Z=Z_1+Z_2$, where $Z_1$ and $Z_2$ are irreducible smooth rational curves.
We~may assume that $Z_2\not\subseteq\mathrm{Supp}(D)$.
If~$P\in Z_2$, then $1=D\cdot Z_2\geqslant m>1$ by Theorem~\ref{theorem:Skoda}.
This shows that $P\in Z_1$ and $Z_1\subseteq\mathrm{Supp}(D)$. .

Let $d$ be the degree of the curve $R$. Then $Z_1^2=Z_2^2=\frac{2-d}{2}$ and $Z_1\cdot Z_2=\frac{d}{2}$.

We may assume that $C_1=Z_1$. Put $\Delta=a_2C_2+\cdots+a_rC_r$. Then $a_1\leqslant\frac{2}{d}$, since
$$
1=Z_2\cdot D=Z_2\cdot\Big(a_1C_1+\Delta\Big)=a_1 Z_2\cdot C_1+Z_2\cdot\Delta\geqslant a_1 Z_2\cdot C_1=\frac{a_1d}{2}.
$$

Denote by $\widetilde{C}_1$ the proper transform of the curve $C_1$ on the surface $\widetilde{S}$. Then $Q\in\widetilde{C}_1$.
Denote by $\widetilde{\Delta}$ the proper transform of the $\mathbb{Q}$-divisor $\Delta$ on the surface $\widetilde{S}$.
The log pair
$$
\Bigg(\widetilde{S}, a_1\widetilde{C}_1+\widetilde{\Delta}+\Big(a_1+\mathrm{mult}_P\big(\Delta\big)-1\Big)F\Bigg)
$$
is not log canonical at the point $Q$ by construction. By Theorem~\ref{theorem:adjunction}, we have
$$
1+\frac{d-2}{2}a_1-\mathrm{mult}_P\big(\Delta\big)=\widetilde{C}_1\cdot\widetilde{\Delta}\geqslant\big(\widetilde{C}_1\cdot\widetilde{\Delta}\big)_Q>1-\Big(a_1+\mathrm{mult}_P\big(\Delta\big)-1\Big),
$$
so that $a_1>\frac{2}{d}$. But we already proved that  $a_1\leqslant\frac{2}{d}$.
\end{proof}

Fix a point $Q\in F$. Put $\widetilde{m}=\mathrm{mult}_{Q}(\widetilde{D})$. Let $g\colon \widehat{S}\to\widetilde{S}$ be a blow up of the point~$Q$.
Denote~by~$\widehat{C}$ and $\widehat{F}$ the proper transforms of the curves $\widetilde{C}$ and $F$ via $g$, respectively.
Similarly, let us denote by $\widehat{D}$ the proper transform of the $\mathbb{Q}$-divisor $D$ on the surface $\widehat{S}$.
Denote by $G$ the $g$-exceptional curve. If the log pair \eqref{equation:log-pair-2} is not log canonical at $Q$, then
\begin{equation}
\label{equation:log-pair-3}
\Bigg(\widehat{S},a\widehat{C}+\widehat{D}+\Big(a\mathrm{mult}_{P}\big(C\big)+m-1\Big)\widehat{F}+\Big(a\mathrm{mult}_{P}\big(C\big)+a\mathrm{mult}_{Q}\big(\widetilde{C}\big)+m+\widetilde{m}-2\Big)G\Bigg)
\end{equation}
is not log canonical at some point in $G$.

\begin{lemma}
\label{lemma:adjunction-2}
Suppose $m\leqslant 1$, $a\mathrm{mult}_{P}(C)+m\leqslant 2$ and $a\mathrm{mult}_{P}\big(C\big)+a\mathrm{mult}_{Q}(\widetilde{C})+2m\leqslant 3$.
Then \eqref{equation:log-pair-3} is log canonical at every point in $G\setminus\widehat{C}$.
\end{lemma}

\begin{proof}
Suppose that \eqref{equation:log-pair-3} is not log canonical at some point $O\in G$ such that $O\not\in\widehat{C}$.
If~$O\not\in\widehat{F}$, then $1\geqslant m\geqslant\widetilde{m}=\widehat{D}\cdot G\geqslant(\widehat{D}\cdot G)_O>1$ by Theorem~\ref{theorem:adjunction}.
Then $O\in\widehat{F}$. Then
$$
m-\widetilde{m}=\big(\widehat{D}\cdot\widehat{F}\big)_O>1-\Big(a\mathrm{mult}_{P}\big(C\big)+a\mathrm{mult}_{Q}\big(\widetilde{C}\big)+m+\widetilde{m}-2\Big)
$$
by Theorem~\ref{theorem:adjunction}. This is impossible, since $a\mathrm{mult}_{P}(C)+a\mathrm{mult}_{Q}(\widetilde{C})+2m\leqslant 3$.
\end{proof}

Fix a point $O\in G$. Put $\widehat{m}=\mathrm{mult}_{O}(\widehat{D})$.
Let $h\colon\overline{S}\to\widehat{S}$ be a blow up of the point $O$.
Denote~by~$\overline{C}$, $\overline{F}$, $\overline{G}$ the proper transforms of the curves $\widehat{C}$, $\widehat{F}$ and $G$ via $h$, respectively.
Similarly, let us denote by $\overline{D}$ the proper transform of the $\mathbb{Q}$-divisor $D$ on the surface $\overline{S}$.
Let $H$ be the $h$-exceptional curve. If $O=G\cap\widehat{F}$ and \eqref{equation:log-pair-3} is not log canonical at $O$,~then
\begin{multline}
\label{equation:log-pair-4}
\Bigg(\overline{S}, a\overline{C}+\overline{D}+\Big(2a\mathrm{mult}_{P}\big(C\big)+a\mathrm{mult}_{Q}\big(\widetilde{C}\big)+a\mathrm{mult}_{O}\big(\widehat{C}\big)+2m+\widetilde{m}+\widehat{m}-4\Big)H+\\
+\Big(a\mathrm{mult}_{P}\big(C\big)+m-1\Big)\overline{F}+\Big(a\mathrm{mult}_{P}\big(C\big)+a\mathrm{mult}_{Q}\big(\widetilde{C}\big)+m+\widetilde{m}-2\Big)\overline{G}\Bigg)
\end{multline}
is not log canonical at some point in $H$.

\begin{lemma}
\label{lemma:adjunction-4}
Suppose that $O=G\cap\widehat{F}$, $m\leqslant 1$, $a\mathrm{mult}_{P}(C)+a\mathrm{mult}_{Q}(\widetilde{C})+m+\widetilde{m}\leqslant 3$~and
$$
2a\mathrm{mult}_{P}\big(C\big)+a\mathrm{mult}_{Q}\big(\widetilde{C}\big)+a\mathrm{mult}_{O}\big(\widehat{C}\big)+4m\leqslant 5.
$$
Then the log pair \eqref{equation:log-pair-4} is log canonical at every point in $H\setminus\overline{C}$.
\end{lemma}

\begin{proof}
Suppose that the pair \eqref{equation:log-pair-4} is not log canonical at some $E\in H$ such that $E\not\in\overline{C}$.
If~$E\not\in\overline{F}\cup\overline{G}$, then
$m\geqslant\widehat{m}=\overline{D}\cdot H\geqslant(\overline{D}\cdot H)_E>1$ by Theorem~\ref{theorem:adjunction}.
Then $E\in\overline{F}\cup\overline{G}$.

If $E\in\overline{G}$, then $E\not\in\overline{F}$, so that Theorem~\ref{theorem:adjunction}  gives
$$
\widetilde{m}-\widehat{m}=\big(\overline{D}\cdot\overline{F}\big)_E>1-\Big(2a\mathrm{mult}_{P}\big(C\big)+a\mathrm{mult}_{Q}\big(\widetilde{C}\big)+a\mathrm{mult}_{O}\big(\widehat{C}\big)+2m+\widetilde{m}+\widehat{m}-4\Big),
$$
which is impossible, since $2a\mathrm{mult}_{P}(C)+a\mathrm{mult}_{Q}(\widetilde{C})+a\mathrm{mult}_{O}(\widehat{C})+4m\leqslant 5$ by assumption.
Similarly, if $E\in\overline{F}$, then $E\not\in\overline{G}$, so that Theorem~\ref{theorem:adjunction}  gives
$$
m-\widetilde{m}-\widehat{m}=\big(\overline{D}\cdot\overline{F}\big)_E>1-\Big(2a\mathrm{mult}_{P}\big(C\big)+a\mathrm{mult}_{Q}\big(\widetilde{C}\big)+a\mathrm{mult}_{O}\big(\widehat{C}\big)+2m+\widetilde{m}+\widehat{m}-4\Big),
$$
which is impossible, since  $2a\mathrm{mult}_{P}(C)+a\mathrm{mult}_{Q}(\widetilde{C})+a\mathrm{mult}_{O}(\widehat{C})+4m\leqslant 5$ .
\end{proof}

Let $Z$ be an irreducible curve in $S$ such that $P\in Z$. Suppose also that $Z\not\subseteq\mathrm{Supp}(D)$.
Denote its proper transforms on the surfaces $\widetilde{S}$ and $\widehat{S}$ by the symbols $\widetilde{Z}$ and $\widehat{Z}$, respectively.
Fix $b\in\mathbb{Q}$ such that $0\leqslant b\leqslant 1$. If $(S, aC+bZ+D)$ is not log canonical at $P$, then
\begin{equation}
\label{equation:log-pair-5}
\Bigg(\widetilde{S},a\widetilde{C}+b\widetilde{Z}+\widetilde{D}+\Big(a\mathrm{mult}_{P}\big(C\big)+b\mathrm{mult}_{P}\big(Z\big)+m-1\Big)F\Bigg)
\end{equation}
is not log canonical at some point in $F$.

\begin{lemma}
\label{lemma:adjunction-7}
Suppose that $m\leqslant 1$ and
$$
a\mathrm{mult}_{P}\big(C\big)+b\mathrm{mult}_{P}\big(Z\big)+m\leqslant 2.
$$
Then \eqref{equation:log-pair-5} is log canonical at every point in $Q\in F\setminus(\widetilde{C}\cup\widetilde{Z})$.
\end{lemma}

\begin{proof}
Suppose that \eqref{equation:log-pair-5} is not log canonical at some point $Q\in F$ such that $Q\not\in\widetilde{C}\cup\widetilde{Z}$.
Then $m=\widetilde{D}\cdot F\geqslant(\widetilde{D}\cdot F)_Q>1$ by Theorem~\ref{theorem:adjunction}. But $m\leqslant 1$ by assumption.
\end{proof}

If the log pair \eqref{equation:log-pair-5} is not log canonical at $Q$, then the log pair
\begin{multline}
\label{equation:log-pair-6}
\Bigg(\widehat{S},a\widehat{C}+b\widehat{Z}+\widehat{D}+\Big(a\mathrm{mult}_{P}\big(C\big)+b\mathrm{mult}_{P}\big(Z\big)+m-1\Big)F+\\
+\Big(a\mathrm{mult}_{P}\big(C\big)+a\mathrm{mult}_{Q}\big(\widetilde{C}\big)+b\mathrm{mult}_{P}\big(Z\big)+b\mathrm{mult}_{Q}\big(\widetilde{Z}\big)+m+\widetilde{m}-2\Big)G\Bigg)
\end{multline}
is not log canonical at some point in $G$.

\begin{lemma}
\label{lemma:adjunction-8}
Suppose that $m\leqslant 1$, $a\mathrm{mult}_{P}(C)+b\mathrm{mult}_{P}(Z)+m\leqslant 2$ and
$$
a\mathrm{mult}_{P}\big(C\big)+a\mathrm{mult}_{Q}\big(\widetilde{C}\big)+b\mathrm{mult}_{P}\big(Z\big)+b\mathrm{mult}_{Q}\big(\widetilde{Z}\big)+2m\leqslant 3.
$$
Then the log pair \eqref{equation:log-pair-6} is log canonical at every point in $G\setminus(\widehat{C}\cup\widehat{Z})$.
\end{lemma}

\begin{proof}
We may assume that the log pair \eqref{equation:log-pair-6} is not log canonical at $O$ and $O\not\in\widehat{C}\cup\widehat{Z}$.
If $O\not\in\widehat{F}$, then $m\geqslant\widetilde{m}=\widehat{D}\cdot G\geqslant(\widehat{D}\cdot G)_O>1$
by Theorem~\ref{theorem:adjunction}, so that $O\in\widehat{F}$. Then
$$
m-\widetilde{m}=\big(\widehat{D}\cdot\widehat{F}\big)_O>1-\Big(a\mathrm{mult}_{P}\big(C\big)+a\mathrm{mult}_{Q}\big(\widetilde{C}\big)+b\mathrm{mult}_{P}\big(Z\big)+b\mathrm{mult}_{Q}\big(\widetilde{Z}\big)+m+\widetilde{m}-2\Big),
$$
by Theorem~\ref{theorem:adjunction}, so that $a\mathrm{mult}_{P}\big(C\big)+a\mathrm{mult}_{Q}(\widetilde{C})+b\mathrm{mult}_{P}(Z)+b\mathrm{mult}_{Q}(\widetilde{Z})+2m>3$.
\end{proof}

\section{Eight local lemmas}
\label{section:eight-lemmas}

Let us use notations and assumptions of Section~\ref{section:pairs}. Fix $x\in\mathbb{Q}$ such that $0\leqslant x\leqslant 1$.~Put
$$
\mathrm{lct}_{P}\big(S,C\big)=\mathrm{sup}\Big\{\lambda\in\mathbb{Q}\ \Big|\ \text{the log pair}\ \big(S, \lambda C\big)\ \text{is log canonical at}\ P\Big\}\in\mathbb{Q}_{>0}.
$$

\begin{lemma}
\label{lemma:local-1}
Suppose that $C$ has an ordinary node or an ordinary cusp at $P$, $a\leqslant\frac{x}{2}$ and
$$
\big(D\cdot C\big)_{P}\leqslant\frac{4}{3}+\frac{x}{6}-a.
$$
Then the log pair $(S, aC+D)$ is log canonical at $P$.
\end{lemma}

\begin{proof}
We have $2m\leqslant\mathrm{mult}_{P}(D)\mathrm{mult}_{P}(C)\leqslant(D\cdot C)_{P}\leqslant \frac{4}{3}+\frac{x}{6}-a$,
so that $2m+a\leqslant \frac{4}{3}+\frac{x}{6}$.
Then $m\leqslant\frac{3}{4}$ and $m+2a=m+\frac{a}{2}+\frac{3a}{2}\leqslant\frac{\frac{4}{3}+\frac{x}{6}}{2}+\frac{3a}{2}\leqslant\frac{\frac{4}{3}+\frac{x}{6}}{2}+\frac{3x}{4}=\frac{2}{3}+\frac{5}{6}x\leqslant\frac{3}{2}$.

Suppose that $(S, aC+D)$ is not log canonical at $P$. Let us seek for a contradiction.
We may assume that \eqref{equation:log-pair-2} is not log canonical at $Q$.
Then $Q\in\widetilde{C}$ by Corollary~\ref{corollary:adjunction}.~Then
$$
\big(\widetilde{D}\cdot\widetilde{C}\big)_O>1-\big(2a+m-1\big)\big(\widetilde{C}\cdot F\big)_O\geqslant 1-2\big(2a+m-1\big)=3-4a-2m.
$$
On the other hand, we have
$\frac{4}{3}+\frac{x}{6}-a\geqslant(D\cdot C)_{P}\geqslant 2m+(\widetilde{D}\cdot\widetilde{C})_O$,
so that $a>\frac{5}{9}-\frac{x}{18}$. Then $\frac{x}{2}\geqslant a>\frac{5}{9}-\frac{x}{18}$,
so that $x>1$. But $x\leqslant 1$ by assumption.
\end{proof}

\begin{lemma}
\label{lemma:local-2}
Suppose that $C$ has an ordinary node or an ordinary cusp at $P$, and
$$
\big(D\cdot C\big)_{P}\leqslant\mathrm{lct}_{P}(S,C)+\frac{x}{2}.
$$
Suppose also that $a\leqslant\mathrm{lct}_{P}(S,C)-\frac{x}{2}$.
Then $(S, aC+D)$ is log canonical at $P$.
\end{lemma}

\begin{proof}
We have $2m\leqslant(D\cdot C)_{P}$. This gives $2m+a\leqslant 1+\frac{x}{2}$. Thus, we have $m\leqslant\frac{1+\frac{x}{2}}{2}\leqslant\frac{3}{4}$.
Similarly, we get $m+2a=m+\frac{a}{2}+\frac{3a}{2}\leqslant\frac{1+\frac{x}{2}}{2}+\frac{3a}{2}\leqslant\frac{1+\frac{x}{2}}{2}+\frac{3}{2}(1-\frac{x}{2})=2-\frac{x}{2}\leqslant 2$.

Suppose that  $(S, aC+D)$ is not log canonical at $P$. Let us seek for a contradiction.
We~may assume that the pair \eqref{equation:log-pair-2} is not log canonical at $Q$.
Then $Q\in\widetilde{C}$ by Corollary~\ref{corollary:adjunction}.
We may assume that \eqref{equation:log-pair-3} is not log canonical at $O$.
Then $O\in\widehat{C}$ by Lemma~\ref{lemma:adjunction-2}, since
$$
3a+2m\leqslant 2a+1+\frac{x}{2}\leqslant 2-x+1+\frac{x}{2}=3-\frac{x}{2}\leqslant 3,
$$
because $2m+a\leqslant 1+\frac{x}{2}$ and $a\leqslant 1-\frac{x}{2}$.
If $O\not\in\widehat{F}$, then Theorem~\ref{theorem:adjunction}  gives
$$
1+\frac{x}{2}-a\geqslant\big(D\cdot C\big)_{P}-2m-\widetilde{m}\geqslant\big(\widehat{D}\cdot\widehat{C}\big)_O>1-\big(3a+m+\widetilde{m}-2\big),
$$
which implies that $2a+\frac{x}{2}>2+m$.
But $2a+\frac{x}{2}\leqslant 2-\frac{x}{2}$, because $a\leqslant \mathrm{lct}_{P}(S,C)-\frac{x}{2}\leqslant 1-\frac{x}{2}$.
This shows that $O=G\cap\widehat{F}\cap\widehat{C}$.
In particular, the curve $C$ has an ordinary cusp at $P$.
By assumption, we have $a\leqslant\frac{5}{6}-\frac{x}{2}$ and $2m+a\leqslant\frac{5}{6}+\frac{x}{2}$.
This gives $6a+4m\leqslant 5-x\leqslant 5$.

Put $E=H\cap\overline{C}$. Then \eqref{equation:log-pair-4} is not log canonical at $E$ by Lemma~\ref{lemma:adjunction-4}.~Then
$$
\big(\overline{D}\cdot\overline{C}\big)_E>1-\Big(6a+2m+\widetilde{m}+\widehat{m}-4\Big)=5-6a-2m-\widetilde{m}-\widehat{m}
$$
by Theorem~\ref{theorem:adjunction}.
Thus, we have $\frac{5}{6}+\frac{x}{2}-a\geqslant(D\cdot C)_{P}\geqslant 2m+\widetilde{m}+\widehat{m}+(\overline{D}\cdot\overline{C})_E>5-6a$.
This gives $a>\frac{5}{6}-\frac{x}{10}$. But $a\leqslant\frac{5}{6}-\frac{x}{2}$, which is absurd.
\end{proof}

\begin{lemma}
\label{lemma:local-3}
Suppose that $C$ is smooth at $P$, $a\leqslant\frac{1}{3}+\frac{x}{2}$, $m+a\leqslant 1+\frac{x}{2}$ and
$$
\big(D\cdot C\big)_{P}\leqslant 1-\frac{x}{2}+a.
$$
Then the log pair $(S, aC+D)$ is log canonical at $P$.
\end{lemma}

\begin{proof}
We have $m\leqslant(D\cdot C)_{P}$, so that $m-a\leqslant 1-\frac{x}{2}$.
Then $m\leqslant 1$, since $m+a\leqslant 1+\frac{x}{2}$.

Suppose that $(S, aC+D)$ is not log canonical at $P$. Let us seek for a contradiction.
We~may assume that the pair \eqref{equation:log-pair-2} is not log canonical at $Q$.
Then $Q\in\widetilde{C}$ by Corollary~\ref{corollary:adjunction}.
We~may assume that \eqref{equation:log-pair-3} is not log canonical at $O$.
Then $O\in\widehat{C}$ by Lemmas~\ref{lemma:adjunction-2}. Then
$$
\big(\widehat{D}\cdot\widehat{C}\big)_O>1-\big(2a+m+\widetilde{m}-2\big)=3-2a-m-\widetilde{m}
$$
by Theorem~\ref{theorem:adjunction}.
Then
$1-\frac{x}{2}+a\geqslant(D\cdot C)_{P}\geqslant m+(\widetilde{D}\cdot\widetilde{C})_{Q}\geqslant m+\widetilde{m}+(\widehat{D}\cdot\widehat{C})_O>3-2a$.
This give $a>\frac{2}{3}+\frac{x}{6}$, which is impossible, since $a\leqslant\frac{1}{3}+\frac{x}{2}$ and $x\leqslant 1$.
\end{proof}

\begin{lemma}
\label{lemma:local-4}
Suppose that $C$ is smooth at $P$, $a\leqslant\frac{8}{9}-\frac{x}{18}$, $m+a\leqslant\frac{4}{3}+\frac{x}{6}$ and
$$
\big(D\cdot C\big)_{P}\leqslant\frac{x}{2}+a.
$$
Then the log pair $(S, aC+D)$ is log canonical at $P$.
\end{lemma}

\begin{proof}
We have $m\leqslant(D\cdot C)_{P}$, so that $m-a\leqslant\frac{x}{2}$.
Then $m\leqslant\frac{2}{3}+\frac{x}{3}\leqslant 1$, since $m+a\leqslant\frac{4}{3}+\frac{x}{6}$.

Suppose that $(S, aC+D)$ is not log canonical at $P$. Let us seek for a contradiction.
We~may assume that the pair \eqref{equation:log-pair-2} is not log canonical at $Q$.
Then $Q\in\widetilde{C}$ by Corollary~\ref{corollary:adjunction}.
We~may assume that \eqref{equation:log-pair-3} is not log canonical at $O$. Then $O\in\widehat{C}$ by Lemmas~\ref{lemma:adjunction-2}. Then
$$
\big(\widehat{D}\cdot\widehat{C}\big)_O>1-\big(2a+m+\widetilde{m}-2\big)=3-2a-m-\widetilde{m}
$$
by Theorem~\ref{theorem:adjunction}.
Then $\frac{x}{2}+a\geqslant(D\cdot C)_{P}\geqslant m+(\widetilde{D}\cdot\widetilde{C})_{Q}\geqslant m+\widetilde{m}+(\widehat{D}\cdot\widehat{C})_O>3-2a$.
This gives $a>1-\frac{x}{6}$, which is impossible, since $a\leqslant\frac{8}{9}-\frac{x}{18}$ and $x\leqslant 1$.
\end{proof}

\begin{lemma}
\label{lemma:local-5}
Suppose that $C$ has an ordinary node or an ordinary cusp at $P$, $a\leqslant \frac{1+x}{3}$~and
$$
\big(D\cdot C\big)_{P}\leqslant 2-2a.
$$
Then the log pair $(S, aC+D)$ is log canonical at $P$.
\end{lemma}

\begin{proof}
We have $2m\leqslant(D\cdot C)_{P}\leqslant 2-2a$. This gives $m+a\leqslant 1$, so that we have $m\leqslant 1$.
Then $m+2a\leqslant 1+a\leqslant 1+\frac{1+x}{3}=\frac{4+x}{3}\leqslant\frac{5}{3}$
and $3a+2m\leqslant 2+a\leqslant 2+\frac{1+x}{3}=\frac{7+x}{3}\leqslant\frac{8}{3}$.

Suppose that $(S, aC+D)$ is not log canonical at $P$. Let us seek for a contradiction.
We~may assume that the pair \eqref{equation:log-pair-2} is not log canonical at $Q$.
Then $Q\in\widetilde{C}$ by Corollary~\ref{corollary:adjunction}.
We may assume that \eqref{equation:log-pair-3} is not log canonical at $O$.
Then $O\in\widehat{C}$ by Lemma~\ref{lemma:adjunction-2}.

If $O\not\in\widehat{F}$, then $(\widehat{D}\cdot\widehat{C})_O>3-3a-m-\widetilde{m}$ by Theorem~\ref{theorem:adjunction},
so that
$$
2-2a\geqslant\big(D\cdot C\big)_{P}\geqslant 2m+\big(\widetilde{D}\cdot\widetilde{C}\big)_{Q}\geqslant 2m+\widetilde{m}+\big(\widehat{D}\cdot\widehat{C}\big)_O>3-3a,
$$
which is absurd. This~shows that $O=G\cap\widehat{F}\cap\widehat{C}$. Then
$$
\big(\widehat{D}\cdot\widehat{C}\big)_O>1-\big(2a+m-1\big)-\big(3a+m+\widetilde{m}-2\big)=4-5a-2m-\widetilde{m}
$$
by Theorem~\ref{theorem:adjunction}. Then
$2-2a\geqslant(D\cdot C)_{P}\geqslant 2m+\widetilde{m}+(\widehat{D}\cdot\widehat{C})_O>4-5a$,
so that $a>\frac{2}{3}$. But $a\leqslant \frac{1+x}{3}\leqslant\frac{2}{3}$ by assumption. This is a contradiction.
\end{proof}

\begin{lemma}
\label{lemma:local-6}
Suppose that $C$ has an ordinary node or an ordinary cusp at $P$, $a\leqslant\frac{2}{3}$ and
$$
\big(D\cdot C\big)_{P}\leqslant\frac{4}{3}+\frac{2x}{3}-2a.
$$
Then the log pair $(S, aC+D)$ is log canonical at $P$.
\end{lemma}

\begin{proof}
We have $2m\leqslant(D\cdot C)_{P}$,
so that $m+a\leqslant\frac{2}{3}+\frac{x}{3}\leqslant 1$. Then $m\leqslant 1$ and
$m+2a\leqslant\frac{5}{3}$.
Similarly, we see that $3a+2m\leqslant\frac{4}{3}+\frac{2x}{3}+a\leqslant\frac{4}{3}+\frac{2x}{3}+\frac{2}{3}=2+\frac{2x}{3}\leqslant\frac{8}{3}<3$.

Suppose that $(S, aC+D)$ is not log canonical at $P$. Let us seek for a contradiction.
We~may assume that the pair \eqref{equation:log-pair-2} is not log canonical at $Q$.
Then $Q\in\widetilde{C}$ by Corollary~\ref{corollary:adjunction}.
We~may assume that \eqref{equation:log-pair-3} is not log canonical at $O$.
Then $O\in\widehat{C}$ by Lemma~\ref{lemma:adjunction-2}.

If $O\not\in\widehat{F}$, then
$\frac{4}{3}+\frac{2x}{3}-2a\geqslant(D\cdot C)_{P}\geqslant 2m+\widetilde{m}+(\widehat{D}\cdot\widehat{C})_O>m+3-3a$ by Theorem~\ref{theorem:adjunction}.
Therefore, if $O\not\in\widehat{F}$, then $a>\frac{5}{3}-\frac{2x}{3}\geqslant 1$.
But $a\leqslant\frac{2}{3}$. This shows that $O=G\cap\widehat{F}\cap\widehat{C}$.
Then~$(\widehat{D}\cdot\widehat{C})_O>1-(2a+m-1)-(3a+m+\widetilde{m}-2)=4-5a-2m-\widetilde{m}$~by~Theorem~\ref{theorem:adjunction}.
Then $\frac{4}{3}+\frac{2x}{3}-2a\geqslant(D\cdot C)_{P}\geqslant 2m+\widetilde{m}+(\widehat{D}\cdot\widehat{C})_O>4-5a$,
which gives $a>\frac{2}{3}$.
\end{proof}

\begin{lemma}
\label{lemma:local-7}
Suppose that $C$ and $Z$ are smooth at $P$, $(C\cdot Z)_P\leqslant 2$, and $a+b+m\leqslant 1+\frac{x}{2}$.
Suppose also that $a\leqslant\frac{1+x}{3}$, $b\leqslant\frac{1+x}{3}$, $(D\cdot C)_{P}\leqslant 1+a-2b$ and $(D\cdot Z)_{P}\leqslant 1+b-2a$.
Then the log pair $(S, aC+bZ+D)$ is log canonical at $P$.
\end{lemma}

\begin{proof}
We have
$m\leqslant(D\cdot C)_{P}\leqslant 1+a-2b$ and $m\leqslant(D\cdot Z)_{P}\leqslant 1+b-2a$.
Then~$m+\frac{a+b}{2}\leqslant 1$.

Suppose that $(S,aC+bZ+D)$ is not log canonical at $P$. Let us seek for a contradiction.
We~may assume that \eqref{equation:log-pair-5} is not log canonical at $Q$.
Then $Q\in\widetilde{C}\cup\widetilde{Z}$ by Lemma~\ref{lemma:adjunction-7}.
Without loss of generality, we may assume that $\widetilde{C}$ contains $Q$. Then $\widetilde{Z}$ also contains $Q$.
Indeed, if $Q\not\in\widetilde{Z}$, then
$1+a-2b\geqslant(D\cdot C)_{P}\geqslant m+(\widetilde{D}\cdot\widetilde{C})_{Q}>2-a-b$
by Theorem~\ref{theorem:adjunction}.
But $1+b-2a\geqslant 0$. Thus, we have $Q=G\cap\widetilde{C}\cap\widetilde{Z}$, so that $(C\cdot Z)_P=2$.

We may assume that \eqref{equation:log-pair-6} is not log canonical at $O$. Then $O\in\widehat{C}\cup\widehat{Z}$ by Lemma~\ref{lemma:adjunction-8}.
In particular, we have $O\not\in\widehat{F}$.
Without loss of generality, we may assume that $O\in\widehat{C}$.
By Theorem~\ref{theorem:adjunction}, we have $1+a-2b-m-\widetilde{m}\geqslant(\widehat{D}\cdot\widehat{C})_O>1-(2a+2b+m+\widetilde{m}-2)$.
This gives $a>\frac{2}{3}$, which is impossible, since $a\leqslant 1+\frac{x}{2}\leqslant\frac{2}{3}$.
\end{proof}

\begin{lemma}
\label{lemma:local-8}
Suppose that $C$ and $Z$ are smooth at $P$, $(C\cdot Z)_P\leqslant 2$, and $a+b+m\leqslant\frac{4}{3}+\frac{x}{6}$.
Suppose also that $a\leqslant\frac{2}{3}$, $b\leqslant\frac{2}{3}$, $(D\cdot C)_{P}\leqslant\frac{2+x}{3}+a-2b$ and $(D\cdot Z)_{P}\leqslant \frac{2+x}{3}+b-2a$.
Then the log pair $(S, aC+bZ+D)$ is log canonical at $P$.
\end{lemma}

\begin{proof}
We have $m\leqslant(D\cdot C)_{P}\leqslant \frac{2+x}{3}+a-2b$
and we have  $m\leqslant(D\cdot Z)_{P}\leqslant\frac{2+x}{3}+b-2a$.
Then $m+\frac{a+b}{2}\leqslant\frac{2+x}{3}\leqslant 1$, $m+a+b\leqslant \frac{4}{3}+\frac{x}{6}\leqslant\frac{3}{2}$ and $2a-b\leqslant 1$.

Suppose that $(S,aC+bZ+D)$ is not log canonical at $P$. Let us seek for a contradiction.
We~may assume that \eqref{equation:log-pair-5} is not log canonical at $Q$.
Then $Q\in\widetilde{C}\cup\widetilde{Z}$ by Lemma~\ref{lemma:adjunction-7}.
Without loss of generality, we may assume that $Q$ is contained in $\widetilde{C}$. Then~$Q\in\widetilde{C}\cap\widetilde{Z}$.
Indeed, if $\widetilde{Z}$ does not contain $Q$, then $\frac{2+x}{3}+a-2b\geqslant m+\big(\widetilde{D}\cdot\widetilde{C}\big)_{Q}>2-a-b$~by~Theorem~\ref{theorem:adjunction}.
The later inequality immediately leads to a contradiction, since $2a-b\leqslant 1$.

We may assume that \eqref{equation:log-pair-6} is not log canonical at $O$.
Then $O\in\widehat{C}\cup\widehat{Z}$ by Lemmas~\ref{lemma:adjunction-8}.
In particular, we have $O\not\in\widehat{F}$. Without loss of generality, we may assume that $O\in\widehat{C}$.
Then $\frac{2+x}{3}+a-2b-m-\widetilde{m}\geqslant(\widehat{D}\cdot\widehat{C})_O>1-(2a+2b+m+\widetilde{m}-2)$ by Theorem~\ref{theorem:adjunction}.
This gives $a>\frac{7-x}{9}$, which is impossible, since $a\leqslant\frac{2}{3}$.
\end{proof}

\section{The proof of main result}
\label{section:main}

Let $S$ be a smooth del Pezzo surface such that $K_S^2=1$.
Then $|-2K_S|$ is base point~free.
It is well-known that the linear system $|-2K_S|$ gives a double cover $S\to\mathbb{P}(1,1,2)$.
This~double cover induces an involution $\tau\in\mathrm{Aut}(S)$.

Let~$C$~be an irreducible curve in $S$ such that $C^2=-1$. Then $-K_{S}\cdot C=1$ and $C\cong\mathbb{P}^1$.
Put $\widetilde{C}=\tau(C)$. Then $\widetilde{C}^2=K_{S}\cdot\widetilde{C}=-1$ and $\widetilde{C}\cong\mathbb{P}^1$.
Moreover, we have $C+\widetilde{C}\sim -2K_S$.
Furthermore, the irreducible curve $\widetilde{C}$ is uniquely determined by this rational equivalence.
Since~$C\cdot(C+\widetilde{C})=-2K_S\cdot C=2$ and $C^2=-1$, we have $C\cdot\widetilde{C}=3$, so that $1\leqslant |C\cap\widetilde{C}|\leqslant 3$.

Fix $\lambda\in\mathbb{Q}$. Then $-K_S+\lambda C$ is ample $\iff$ $-\frac{1}{3}<\lambda<1$. Indeed, we have
\begin{equation}
\label{equation:second-curve}
-K_S+\lambda C\sim_{\mathbb{Q}}\frac{1}{2}\big(C+\widetilde{C}\big)+\lambda C=\Big(\frac{1}{2}+\lambda\Big)C+\frac{1}{2}\widetilde{C}\sim_{\mathbb{Q}}\big(1+2\lambda\big)\Big(-K_S-\frac{\lambda}{1+2\lambda}\widetilde{C}\Big).
\end{equation}
One the other hand, we have $(-K_S+\lambda C)\cdot C=1-\lambda$ and $(-K_S+\lambda C)\cdot\widetilde{C}=1-3\lambda$.

Note that Theorem~\ref{theorem:main} and \eqref{equation:second-curve} imply

\begin{corollary}
\label{corollary:main}
Suppose that $-\frac{1}{3}<\lambda<1$. If $|C\cap\widetilde{C}|\geqslant 2$, then
$$
\alpha\big(S,-K_S+\lambda C\big)=\left\{%
\aligned
&\mathrm{min}\Big(\frac{\alpha(X)}{1+2\lambda},2\Big)\ \mathrm{if}\ -\frac{1}{3}<\lambda<0,\\%
&\mathrm{min}\Big(\alpha(X),\frac{2}{1+2\lambda}\Big)\ \mathrm{if}\ 0\leqslant\lambda<1.\\%
\endaligned\right.
$$
Similarly, if $|C\cap\widetilde{C}|=1$, then
$$
\alpha\big(S,-K_S+\lambda C\big)=\left\{%
\aligned
&\mathrm{min}\Big(\frac{\alpha(X)}{1+2\lambda},\frac{4}{3+3\lambda}\Big)\ \mathrm{if}\ -\frac{1}{3}<\lambda<0,\\%
&\mathrm{min}\Big(\alpha(X),\frac{4}{3+3\lambda}\Big)\ \mathrm{if}\ 0\leqslant\lambda<1.\\%
\endaligned\right.
$$
\end{corollary}

Now let us prove Theorem~\ref{theorem:main}. Suppose that $0\leqslant \lambda<1$. Put
\begin{equation}
\label{equation:inequality}
\mu=\left\{%
\aligned
&\mathrm{min}\Big(\alpha(S),\frac{2}{1+2\lambda}\Big)\ \mathrm{when}\ |C\cap\widetilde{C}|\geqslant 2,\\%
&\mathrm{min}\Big(\alpha(S),\frac{4}{3+3\lambda}\Big)\ \mathrm{when}\ |C\cap\widetilde{C}|=1.\\%
\endaligned\right.
\end{equation}

\begin{lemma}
\label{lemma:1}
One has $\alpha(S,-K_S+\lambda C)\leqslant\mu$.
\end{lemma}

\begin{proof}
Since we have $(\frac{1}{2}+\lambda)C+\frac{1}{2}\widetilde{C}\sim_{\mathbb{Q}} -K_S+\lambda C$,
we see that $\alpha(S,-K_S+\lambda C)\leqslant\frac{2}{1+2\lambda}$.
Similarly, we see that $\alpha(S,-K_S+\lambda C)\leqslant\alpha(S)$.
If $|C\cap\widetilde{C}|=1$, then the log pair
$$
\Bigg(S,\frac{2+4\lambda}{3+3\lambda}C+\frac{2}{4+3\lambda}\widetilde{C}\Bigg)
$$
is not Kawamata log terminal at the point $C\cap\widetilde{C}$, so that $\alpha(S,-K_S+\lambda C)\leqslant\frac{4}{3+3\lambda}$.
\end{proof}

Thus, to complete the proof of Theorem~\ref{theorem:main}, we have to show that $\alpha(S,-K_S+\lambda C)\geqslant\mu$.
Suppose that $\alpha(S,-K_S+\lambda C)<\mu$. Let us seek for a contradiction.

Since $\alpha(S,-K_S+\lambda C)<\mu$, there exists an effective $\mathbb{Q}$-divisor $D$ on $S$ such that
$$
D\sim_{\mathbb{Q}} -K_S+\lambda C,
$$
and $(S,\mu D)$ is not log canonical at some point $P\in S$.

By Lemma~\ref{lemma:convexity} and \eqref{equation:second-curve}, we may assume that $\mathrm{Supp}(D)$ does not contain $C$ or $\widetilde{C}$.
Indeed, one can check that the log pair $(S,\mu(\frac{1}{2}+\lambda)C+\frac{\mu}{2}\widetilde{C})$ is log canonical at $P$.

Let $\mathcal{C}$ be a curve in the pencil $|-K_S|$ that passes through $P$. Then $\mathcal{C}+\lambda C\sim -K_S+\lambda C$.
Moreover, the curve $\mathcal{C}$ is irreducible, and the log pair $(S,\mu \mathcal{C}+\mu\lambda C)$ is log canonical at~$P$.
Thus, we may assume that $\mathrm{Supp}(D)$ does not contain $C$ or $\mathcal{C}$ by Lemma~\ref{lemma:convexity}.

\begin{lemma}
\label{lemma:2}
The curve $\mathcal{C}$ is smooth at the point $P$.
\end{lemma}

\begin{proof}
Suppose that $\mathcal{C}$ is singular at $P$.
If $\mathcal{C}\not\subseteq\mathrm{Supp}(D)$, then Theorem~\ref{theorem:Skoda} gives
$$
1+\lambda=\mathcal{C}\cdot\Big(-K_S+\lambda C\Big)=\mathcal{C}\cdot D\geqslant
\mathrm{mult}_{P}\big(\mathcal{C}\big)\mathrm{mult}_{P}\big(D\big)\geqslant 2\mathrm{mult}_{P}\big(D\big)>\frac{2}{\mu},
$$
which is impossible by \eqref{equation:inequality}.
Thus, we have $\mathcal{C}\subseteq\mathrm{Supp}(D)$. Then $C\not\subseteq\mathrm{Supp}(D)$.

Write $D=\epsilon\mathcal{C}+\Delta$, where $\epsilon$ is a positive rational number, and $\Delta$ is an effective $\mathbb{Q}$-divisor on the surface $S$ whose support does not contain the curves $\mathcal{C}$ and $C$.
Then
$$
1-\lambda=C\cdot\Big(-K_S+\lambda C\Big)=C\cdot D=C\cdot\Big(\epsilon\mathcal{C}+\Delta\Big)=\epsilon+C\cdot\Delta\geqslant \epsilon,
$$
so that $\epsilon\leqslant 1-\lambda$. Similarly, we have
\begin{equation}
\label{equation:2}
1+\lambda-\epsilon=\mathcal{C}\cdot\Delta\geqslant\big(\mathcal{C}\cdot\Delta\big)_{P}.
\end{equation}

We claim that $\lambda\leqslant\frac{1}{2}$. Indeed, suppose that $\lambda>\frac{1}{2}$. Then it follows from \eqref{equation:2} that
$$
\big(\Delta\cdot\mathcal{C}\big)_{P}\leqslant 1+\lambda-\epsilon=\frac{1+2\lambda}{2}\Bigg(\frac{4}{3}+\frac{\frac{4-4\lambda}{1+2\lambda}}{6}-\frac{2}{1+2\lambda}\epsilon\Bigg).
$$
Thus, we can apply Lemma~\ref{lemma:local-1} to the log pair $(S,\frac{2}{1+2\lambda}D)$ with $x=\frac{4-4\lambda}{1+2\lambda}$ and $a=\frac{2}{1+2\lambda}\epsilon$.
This implies that $(S,\frac{2}{1+2\lambda}D)$ is log canonical at $P$, which is impossible, because $\mu\leqslant\frac{2}{1+2\lambda}$. 

If $\mathcal{C}$ has a node at $P$,
then we can apply Lemma~\ref{lemma:local-2} to $(S,D)$ with $x=2\lambda$ and $a=\epsilon$.
This~implies that $(S,D)$ is log canonical, which is absurd, since $\mu\leqslant 1$.

Therefore, the curve $\mathcal{C}$ has an ordinary cusp at $P$ and $\lambda\leqslant\frac{1}{2}$.
Then $\mu\leqslant\alpha(S)=\frac{5}{6}$.
Thus, we can apply Lemma~\ref{lemma:local-1} to the log pair $(S,\frac{5}{6}D)$ with $x=\frac{5}{3}\lambda$ and $a=\frac{5}{6}\epsilon$, since
$$
\big(\Delta\cdot\mathcal{C}\big)_{P}\leqslant\frac{6}{5}\Bigg(\frac{5}{6}+\frac{5}{6}\lambda-\frac{5}{6}\epsilon\Bigg).
$$
This implies that $(S,\frac{5}{6}D)$ is log canonical at $P$, which is impossible, since $\mu\leqslant\frac{5}{6}$.
\end{proof}

The next step in the proof of Theorem~\ref{theorem:main} is

\begin{lemma}
\label{lemma:3}
The point $P$ is not contained in the curve $C$.
\end{lemma}

\begin{proof}
Suppose that $P\in C$. Let us seek for a contradiction.
If $C\not\subseteq\mathrm{Supp}(D)$, then
$$
1-\lambda=C\cdot\Big(-K_S+\lambda C\Big)=C\cdot D\geqslant
\mathrm{mult}_{P}\big(C\big)\mathrm{mult}_{P}\big(D\big)\geqslant\mathrm{mult}_{P}\big(D\big)>\frac{1}{\mu}
$$
by Theorem~\ref{theorem:Skoda}. 
But \eqref{equation:inequality} implies that $\mu>\frac{1}{1-\lambda}$, which is impossible, because $\mu\leqslant 1$.
Therefore, we must have $C\subseteq\mathrm{Supp}(D)$. Then $\mathcal{C}\not\subseteq\mathrm{Supp}(D)$ and also $\widetilde{C}\not\subseteq\mathrm{Supp}(D)$.

Write $D=\epsilon C+\Delta$, where $\epsilon$ is a positive rational number,
and $\Delta$ is an effective divisor whose support does not contain $\mathcal{C}$, $C$ and $\widetilde{C}$.
Then
$1+\lambda-\epsilon=\mathcal{C}\cdot\Delta\geqslant\mathrm{mult}_{P}(\Delta)$.
Similarly, we have $1+3\lambda-3\epsilon=\widetilde{C}\cdot\Delta\geqslant 0$.
Finally, we have $1-\lambda+\epsilon=C\cdot\Delta\geqslant(C\cdot\Delta)_{P}$.

If $\lambda\leqslant\frac{1}{2}$, we can apply Lemma~\ref{lemma:local-3} to the log pair $(S,D)$ with $x=2\lambda$ and $a=\epsilon$.
This~implies that $(S,D)$ is log canonical, which is impossible since $\mu\leqslant 1$.

Therefore, we have $\lambda>\frac{1}{2}$.
Since $\epsilon\leqslant \frac{1}{3}+\lambda$, we have $\frac{2}{1+2\lambda}\epsilon\leqslant\frac{2}{1+2\lambda}(\frac{1}{3}+\lambda)=\frac{8}{9}-\frac{\frac{4-4\lambda}{1+2\lambda}}{18}$.
Since $\epsilon+\mathrm{mult}_{P}(\Delta)\leqslant 1+\lambda$, we have
$\frac{2}{1+2\lambda}\epsilon+\frac{2}{1+2\lambda}\mathrm{mult}_{P}(\Delta)\leqslant \frac{2}{1+2\lambda}(1+\lambda)=\frac{4}{3}+\frac{\frac{4-4\lambda}{1+2\lambda}}{6}$.~But
$$
\big(\Delta\cdot C\big)_{P}\leqslant 1-\lambda+\epsilon=\frac{1+2\lambda}{2}\Bigg(\frac{\frac{4-4\lambda}{1+2\lambda}}{2}+\frac{2}{1+2\lambda}\epsilon\Bigg).
$$
Thus, we can apply Lemma~\ref{lemma:local-4} to the log pair $(S,\frac{2}{1+2\lambda}D)$ with $x=\frac{4-4\lambda}{1+2\lambda}$ and $a=\frac{2}{1+2\lambda}\epsilon$.
This implies that  $(S,\frac{2}{1+2\lambda}D)$ is log canonical at~$P$, which is impossible, since $\mu\leqslant\frac{2}{1+2\lambda}$.
\end{proof}

Let $h\colon S\to\overline{S}$ be the contraction of the curve $C$.
Put $\overline{D}=h(D)$. Then $\overline{D}\sim_{\mathbb{Q}} -K_{\overline{S}}$.
Moreover, it follows from Lemma~\ref{lemma:3} that $(\overline{S},\mu\overline{D})$ is not log canonical at the point $h(P)$.

By construction, the surface $\overline{S}$ is a smooth del Pezzo surface such that $K_{\overline{S}}^2=K_{S}^2+1=2$.
Then $|-K_{\overline{S}}|$ gives a double cover $\pi\colon\overline{S}\to\mathbb{P}^2$ branched in a smooth quartic curve $R_4\subset\mathbb{P}^2$.
By Lemma~\ref{lemma:dP2}, there exists a unique curve $\overline{Z}\in|-K_{\overline{S}}|$ such that $\overline{Z}$ is singular~at~$h(P)$.
Moreover, the log pair $(\overline{S},\overline{Z})$ is not log canonical at the point $h(P)$ by \cite[Theorem~1.12]{CheltsovParkWon1}.
Note that $\pi(\overline{Z})$ is the line in $\mathbb{P}^2$ that is tangent to the curve $R_4$ at the point $\pi\circ h(P)$.

Let $Z$ be the proper transform of the curve $\overline{Z}$ on the surface $S$. Then $h(C)\not\in\overline{Z}$. 
Indeed,~if~$h(C)$ is contained in $\overline{Z}$, then $Z\sim -K_{S}$, which is impossible by Lemma~\ref{lemma:2}.
Thus, we see that $C\cap Z=\varnothing$. Then $Z\sim -K_{S}+C$.

\begin{lemma}
\label{lemma:4}
The curve $Z$ is reducible.
\end{lemma}

\begin{proof}
Suppose that $Z$ is irreducible. Then $Z$ has an ordinary node or ordinary cusp at~$P$.
In fact, if $Z\not\subseteq\mathrm{Supp}(D)$,~then $2=Z\cdot D>\frac{2}{\mu}$ by Theorem~\ref{theorem:Skoda}, which contradicts to~\eqref{equation:inequality}.
Therefore, we have $Z\subseteq\mathrm{Supp}(D)$.
Put $\widetilde{Z}=\tau(Z)$. Then $Z+\widetilde{Z}\sim -4K_{S}$ and
$$
\frac{3\lambda+1}{4}Z+\frac{1-\lambda}{4}\widetilde{Z}\sim_{\mathbb{Q}} \frac{1-\lambda}{4}\Big(Z+\widetilde{Z}\Big)+\lambda Z \sim_{\mathbb{Q}}-K_{S}+\lambda C.
$$
Furthermore, one can show (using Definition~\ref{definition:lct}) that the log pair
$$
\Bigg(S,\mu\frac{3\lambda+1}{4}Z+\mu\frac{1-\lambda}{4}\widetilde{Z}\Bigg)
$$
is log canonical at $P$.
Hence, we may assume that $\widetilde{Z}\not\subseteq\mathrm{Supp}(D)$ by Lemma~\ref{lemma:convexity}.

Write $D=\epsilon Z+\Delta$, where $\epsilon$ is a positive rational number,
and $\Delta$ is an effective $\mathbb{Q}$-divisor on the surface $S$ whose support does not contain $Z$ and $\widetilde{Z}$.
Then $2+4\lambda-6\epsilon=\widetilde{Z}\cdot\Delta\geqslant 0$.
Thus, we have $\epsilon\leqslant\frac{1+2\lambda}{3}$. Finally, we have
$$
2-2\epsilon=Z\cdot\Delta\geqslant\big(Z\cdot\Delta\big)_{P}.
$$
Therefore, if $\lambda\leqslant\frac{1}{2}$, then we can apply Lemma~\ref{lemma:local-5} to  $(S,D)$ with $x=2\lambda$ and $a=\epsilon$.
This~implies that $(S,D)$ is log canonical at $P$. But $\mu\leqslant 1$. Thus, we have $\lambda>\frac{1}{2}$.

We have $\mu\leqslant\frac{2}{1+2\lambda}$. Then $(S,\frac{2}{1+2\lambda}D)$ is not log canonical at~$P$. We have $\frac{2}{1+2\lambda}\epsilon\leqslant\frac{2}{3}$.
Thus, we can apply Lemma~\ref{lemma:local-6} to $(S,\frac{2}{1+2\lambda}D)$ with $x=\frac{4-4\lambda}{1+2\lambda}$
and $a=\frac{2}{1+2\lambda}\epsilon$, because
$$
\big(\Delta\cdot Z\big)_{P}\leqslant\frac{1+2\lambda}{2}\Bigg(\frac{4}{3}+\frac{2\frac{4-4\lambda}{1+2\lambda}}{3}-2\frac{2}{1+2\lambda}\epsilon\Bigg)=2-2\epsilon.
$$
This implies that  $(S,\frac{2}{1+2\lambda}D)$ is log canonical at $P$, which is absurd, since $\mu\leqslant\frac{2}{1+2\lambda}$.
\end{proof}

Since $Z$ is reducible, $Z=Z_1+Z_2$, where $Z_1$ and $Z_2$ are smooth irreducible curves.
Then $Z_1^2=Z_2^2=-1$~and~$Z_1\cdot Z_2=2$.
Moreover, we have $P\in Z_1\cap Z_2$ and $(Z_1\cdot Z_2)_P\leqslant 2$.
Furthermore, we have $Z_1\cap C=\varnothing$ and $Z_2\cap C=\varnothing$.

We have $Z_1\subseteq\mathrm{Supp}(D)$ and $Z_2\subseteq\mathrm{Supp}(D)$.
Indeed, if $Z_1\not\subseteq\mathrm{Supp}(D)$, then
$$
1=Z_1\cdot\Big(-K_S+\lambda C\Big)=Z_1\cdot D\geqslant\mathrm{mult}_{P}\big(Z_1\big)\mathrm{mult}_{P}\big(D\big)\geqslant\mathrm{mult}_{P}\big(D\big)>\frac{1}{\mu}\geqslant 1
$$
by Theorem~\ref{theorem:Skoda}.
This shows that $Z_1\subseteq\mathrm{Supp}(D)$. Similarly, we have $Z_2\subseteq\mathrm{Supp}(D)$.
But
$$
\big(1-\lambda\big)\mathcal{C}+\lambda\big(Z_1+Z_2\big)\sim_{\mathbb{Q}} -K_{S}+\lambda C.
$$
On the other hand, the log pair $(S,\mu(1-\lambda)\mathcal{C}+\mu\lambda(Z_1+Z_2))$ is log canonical at~$P$.
Therefore, we may assume that $\mathcal{C}\not\subseteq\mathrm{Supp}(D)$ by Lemma~\ref{lemma:convexity}.

Put $\widetilde{Z}_1=\tau(Z_1)$ and put $\widetilde{Z}_2=\tau(Z_2)$. Then $Z_1+\widetilde{Z}_1\sim -2K_{S}$ and $Z_2+\widetilde{Z}_2\sim -2K_{S}$.
This~gives $\mathcal{C}\cdot Z_1=\mathcal{C}\cdot Z_2=1$, $Z_1\cdot\widetilde{Z}_1=Z_2\cdot\widetilde{Z}_2=3$, $Z_1\cdot\widetilde{Z}_2=Z_2\cdot\widetilde{Z}_1=0$, $\widetilde{Z}_1\cdot C=\widetilde{Z}_2\cdot C=2$.
Moreover, we have $Z_1+Z_2\sim -K_{S}+C$. Then
$$
\frac{1+\lambda}{2}Z_1+\lambda Z_2+\frac{1-\lambda}{2}\widetilde{Z}_1\sim_{\mathbb{Q}}\frac{1-\lambda}{2}\Big(Z_1+\widetilde{Z}_1\Big)+\lambda\Big(Z_1+Z_2\Big)\sim_{\mathbb{Q}} -K_{S}+\lambda C
$$
Note that $P\not\in\widetilde{Z}_1$, because $P\in Z_2$ and $\widetilde{Z}_1\cdot Z_2=0$.
Using this, we see that the log pair
$$
\Bigg(S,\mu\frac{1+\lambda}{2}Z_1+\mu\lambda Z_2+\mu\frac{1-\lambda}{2}\widetilde{Z}_1\Bigg)
$$
is log canonical at the point $P$. Hence, we may assume that $\widetilde{Z}_1\not\subseteq\mathrm{Supp}(D)$ by Lemma~\ref{lemma:convexity}.
Similarly, we may assume that $\widetilde{Z}_2\not\subseteq\mathrm{Supp}(D)$ using Lemma~\ref{lemma:convexity} one more time.

Now let us write $D=\epsilon_1 Z_1+\epsilon_2 Z_2+\Delta$, where $\epsilon_1$ and $\epsilon_2$ are positive rational numbers,
and $\Delta$ is an effective divisor whose support does not contain $Z_1$ and $Z_2$.
Then 
$$
1+\lambda-\epsilon_1-\epsilon_2=\mathcal{C}\cdot\Delta\geqslant\mathrm{mult}_{P}\big(\Delta\big).
$$
This gives $\epsilon_1+\epsilon_2+\mathrm{mult}_{P}(\Delta)\leqslant 1+\lambda$. 
We also have $\epsilon_1\leqslant\frac{1+2\lambda}{3}$, since
$$
1+2\lambda-3\epsilon_1=\widetilde{Z}_1\cdot\Delta\geqslant 0.
$$
Similarly, see that $\epsilon_2\leqslant\frac{1+2\lambda}{3}$. Moreover, we have 
$$
1+\epsilon_1-2\epsilon_2=Z_1\cdot\Delta\geqslant\big(Z_1\cdot\Delta\big)_{P}.
$$ 
Finally, we have 
$$
1+\epsilon_2-2\epsilon_1=Z_2\cdot\Delta\geqslant\big(Z_2\cdot\Delta\big)_{P}.
$$

Thus, if $\lambda\leqslant\frac{1}{2}$, then we can apply Lemma~\ref{lemma:local-7} to $(S,D)$ with $x=2\lambda$, $a=\epsilon_1$ and $b=\epsilon_1$.
This implies that $(S,D)$ is log canonical at $P$, which is absurd. Hence, we have $\lambda>\frac{1}{2}$.

Since $\lambda>\frac{1}{2}$, we have $\mu\leqslant\frac{2}{1+2\lambda}$. Then the log pair $(S,\frac{2}{1+2\lambda}D)$ is not log canonical at~$P$.
On the other hand, we have $\frac{2}{1+2\lambda}\epsilon_1\leqslant\frac{2}{3}$ and $\frac{2}{1+2\lambda}\epsilon_2\leqslant\frac{2}{3}$.
We also have
$$
\frac{2}{1+2\lambda}\epsilon_1+\frac{2}{1+2\lambda}\epsilon_2+\frac{2}{1+2\lambda}\mathrm{mult}_{P}(\Delta)\leqslant\frac{2}{1+2\lambda}\big(1+\lambda\big)=\frac{2}{1+2\lambda}+\lambda\frac{2}{1+2\lambda}=\frac{4}{3}+\frac{\frac{4-4\lambda}{1+2\lambda}}{6},
$$
Moreover, we have
$$
\big(\Delta\cdot Z_1\big)_{P}\leqslant1+\epsilon_1-2\epsilon_2=\frac{1+2\lambda}{2}\Bigg(\frac{2}{3}+\frac{\frac{4-4\lambda}{1+2\lambda}}{3}+\frac{2}{1+2\lambda}\epsilon_1-2\frac{2}{1+2\lambda}\epsilon_2\Bigg).
$$
Furthermore, we also have
$$
\big(\Delta\cdot Z_2\big)_{P}\leqslant 1+\epsilon_1-2\epsilon_2=\frac{1+2\lambda}{2}\Bigg(\frac{2}{3}+\frac{\frac{4-4\lambda}{1+2\lambda}}{3}+\frac{2}{1+2\lambda}\epsilon_2-2\frac{2}{1+2\lambda}\epsilon_1\Bigg).
$$
Thus, we can apply Lemma~\ref{lemma:local-8} to $(S,\frac{2}{1+2\lambda}D)$ with $x=\frac{4-4\lambda}{1+2\lambda}$, $a=\frac{2}{1+2\lambda}\epsilon_1$ and $b=\frac{2}{1+2\lambda}\epsilon_2$.
This implies that $(S,\frac{2}{1+2\lambda}D)$ is log canonical at $P$, which is absurd.

The obtained contradiction completes the proof of Theorem~\ref{theorem:main}.

\end{document}